\documentclass[11pt,reqno]{amsart}
\usepackage{geometry}                
\usepackage{amssymb}
\usepackage{epsfig}
\usepackage{graphics,color}

\input{epsf.sty}
\newcommand\vpil[1]{\overset{\leftarrow}{#1}}
\newcommand\hpil[1]{\overset{\rightarrow}{#1}}

\def\notto{\mathrel{\not\to}}

\def\X{{X}}
\def\Y{{Y}}

\newcommand\fpn{\ensuremath{f_n(p)}}
\newcommand\gpn{\ensuremath{g_n(p)}}
\newcommand\gnp{\ensuremath{G(n,p)}}
\newcommand\gnm{\ensuremath{G(n,m)}}
\newcommand\ggnm{{\ensuremath{\vec G(n,m)}}}
\newcommand\ggnp{{\ensuremath{\vec G(n,p)}}}
\newcommand\ntoo{\ensuremath{{n\to\infty}}}
\newcommand\nn{\binom n2}
\newcommand\oi{[0,1]}
\newcommand\hoi{(0,1]}
\newcommand\pc{p_{\mathsf c}}
\newcommand\Cov{\operatorname{Cov}}
\renewcommand{\=}{:=}
\newcommand\lrpar[1]{\left(#1\right)}
\newcommand\expQ[1]{e^{#1}}
\newcommand\Var{\operatorname{Var}}

\newcommand\cO{\mathcal O}
\newcommand\cI{\mathcal I}
\newcommand\cE{\mathcal E}
\newcommand\cX{\mathcal X}
\newcommand\cY{\mathcal Y}
\newcommand\cS{\mathcal S}
\newcommand\set[1]{\ensuremath{\{#1\}}}
\newcommand\bigpar[1]{\bigl(#1\bigr)}
\newcommand\xpar[1]{(#1)}
\newcommand\Bigpar[1]{\Bigl(#1\Bigr)}
\newcommand\Bigparfrac[2]{\Bigpar{\frac{#1}{#2}}}

\newcommand\opnn{o\bigpar{(1-p/2)^{2n}}}

\newcommand\E{\operatorname{\mathbb E{}}}
\renewcommand\P{\operatorname{\mathbb P{}}}
\newcommand\expBig[1]{\exp\Bigl(#1\Bigr)}
\newcommand\expbig[1]{\exp\bigl(#1\bigr)}

\newcommand\explr[1]{\exp\left(#1\right)}

\newcommand\xfrac[2]{#1/#2}
\newcommand\punkt[1]{\if.#1\else.\spacefactor1000\fi{#1}}

\newcommand\ie{i.e\punkt}

\newcommand{\tend}{\longrightarrow}
\newcommand\dto{\overset{\mathrm{d}}{\tend}}
\newcommand\pto{\overset{\mathrm{p}}{\tend}}
\newcommand\parfrac[2]{\lrpar{\frac{#1}{#2}}}
\newcommand\Op{O_{\mathrm p}}
\newcommand\op{o_{\mathrm p}}

\newcommand{\refand}[2]{\ref{#1} and~\ref{#2}}

\newcommand{\maple}{\texttt{Maple}}

\newcommand\qq{^{1/2}}

\newcommand\qqw{^{-1/2}}

\newcommand\qw{^{-1}}
\newcommand\qww{^{-2}}

\newcommand\gl{\lambda}
\newcommand\Bin{\operatorname{Bin}}

\newcommand\hp{\hat p}
\newcommand\pp{\frac{p(1-p)}{(2-p)^2}}
\newcommand\hpp{\frac{\hp(1-\hp)}{(2-\hp)^2}}

\begingroup
  \divide\count255 by 60
  \multiply\count255 by -60
  \advance\count255 by \time
  \ifnum \count255 < 10 \xdef\klockan{\the\count1.0\the\count255}
  \else\xdef\klockan{\the\count1.\the\count255}\fi
\endgroup

\newcommand{\refT}[1]{Theorem~\ref{#1}}

\newcommand{\refL}[1]{Lemma~\ref{#1}}
\newcommand{\refR}[1]{Remark~\ref{#1}}

\newenvironment{romenumerate}[1][0pt]{
\addtolength{\leftmargini}{#1}\begin{enumerate}
 }{\end{enumerate}}

\newtheorem{theorem}{Theorem}[section]
\newtheorem{lemma}[theorem]{Lemma}

\newtheorem{fact}[theorem]{Fact}

\newtheorem{remark}[theorem]{Remark}

\newtheorem{problem}[theorem]{Problem}

\newtheorem{conjecture}[theorem]{Conjecture}

\begin{document}

\title[Correlations for paths in random orientations]
{Correlations for paths in\\ random orientations of $G(n,p)$ and $G(n,m)$}

\author{Sven Erick Alm} 
\address{Department of Mathematics, Uppsala University, 
P.O.\ Box 480,  SE-751 06, Uppsala, Sweden.}
\email{sea@math.uu.se}

\author{Svante Janson} 
\address{Department of Mathematics, Uppsala University, 
P.O.\ Box 480,  SE-751 06, Uppsala, Sweden.}
\email{svante@math.uu.se}

\author{Svante Linusson} \thanks{Svante Linusson is a Royal Swedish Academy of Sciences Research Fellow supported by a grant from the Knut and Alice Wallenberg Foundation.}
\address{Department of Mathematics, KTH-Royal Institute of Technology, 
  SE-100 44, Stockholm, Sweden.}
\email{linusson@math.kth.se}

\thanks{This research was conducted when the authors visited Institut Mittag-Leffler (Djursholm, Sweden). }
\date{8 June, 2010}  

\begin{abstract}
We study random graphs, both $G(n,p)$ and $G(n,m)$, with random
orientations on the edges. For three fixed distinct vertices $s,a,b$ we
study the 
correlation, in the combined probability space, of the events $\{a\to s\}$
and $\{s\to b\}$. 

For $\gnp$, we prove that 
there is a $\pc=1/2$ such that
for a fixed $p<\pc$ the correlation is negative for large enough $n$ and for
$p>\pc$ the correlation is positive for large enough $n$. 
We conjecture that for a fixed $n\ge 27$ the
correlation changes sign three times for three critical values of $p$. 

For $G(n,m)$ it is similarly proved that, with $p=m/\binom{n}{2}$, 
there is a critical
$\pc$ that is the
solution to a certain equation and approximately equal to 0.7993. A lemma,
which computes the probability of non existence of any $\ell$ directed edges
in $\gnm$, is thought to be of independent interest. 

We present exact recursions to compute $\P(a\to s)$ and $\P(a\to s, s\to
b)$. We also briefly discuss the corresponding question in the quenched
version of the problem. 
\end{abstract}

\maketitle

\section{Introduction} \label{S:Intro}
For a graph $G=(V,E)$ we orient each edge with equal probability for the two possible directions and independent of all other edges. 
This model has been studied previously in for instance \cite{AL1, G00,
  SL2}. Let $s,a,b$ be fixed distinct vertices in the graph. 
Let $\{a\to s\}$ denote the event that there exists a directed path from $a$ to $s$. 
The object of this paper is to study the correlation of the two events $\{a\to s\}$ and 
$\{s\to b\}$ in a random graph. 
One might intuitively guess that they are negatively correlated in any graph, that is $\P(a\to s, s\to b) - \P(a\to s) \P(s\to b)< 0$. 
This is however not true for all graphs. 
In fact, the smallest counter example is the graph on four vertices with all edges except $\{a,b\}$. 
In \cite{AL1} it was proved that for the complete graph $K_n$ they are negatively correlated for $n=3$, independent for $n=4$ 
and positively correlated for $n\ge 5$. 
The results in \cite{AL1} suggested that the correlation seemed more likely to be positive in dense graphs, which led us to further investigate the correlation in $G(n,p)$ and in $G(n,m)$ as presented in the present paper. 

 There are two different ways of combining the two random processes. In this paper we deal with the annealed case, which means that we consider the joined probability space of $G(n,p)$ (or $G(n,m)$) and that of edge orientations. We then consider correlations of events in this space. The other possibility is the quenched version, in which one for each graph in $G(n,p)$ (or $G(n,m)$)
computes the correlation of the two events in the probability space of all
edge orientations in that graph. Then one takes the expected value over all
graphs. We discuss briefly 
the quenched version 
in Section \ref{S:Quenched}.
 
\bigskip
 In the random graph $G(n,p)$ there are $n$ vertices and every edge exists
 with probability $p$, $0\le p\le 1$, independently of other edges. Let
 $\ggnp$ be a directed graph obtained by giving each edge in $\gnp$ a random
 direction. 
Our main result in Section \ref{S:Gnp} implies that for a fixed $p$ the events are positively correlated when $p>1/2$ 
for large enough $n$, but negatively correlated when $p<1/2$ for large enough $n$. 
To be more precise, in Theorem \ref{T:korrp} we prove that 
\[ \frac{\P(a\notto s, s\notto b)-\P(a\notto s)\P(s\notto b)}{\P(a\notto s, s\notto b)}\to \frac{2p-1}{3} \text{ as }n\to \infty.\] 
Note that the covariance is bilinear so that $\P(a\to s, s\to b) - \P(a\to s) \P(s\to b) = \P(a\notto s, s\notto b) - \P(a\notto s) \P(s\notto b)$.  
Thus we may instead study the complementary events $\{a\notto s\}$, that there does not exist a directed path from $a$ to $s$,  
and $\{s\notto b\}$, which turns out to be more practical. The change of sign for the correlation at $p=1/2$ was unexpected for us. An attempt to an heuristic explanation is given just after the statement of Theorem \ref{T:korrp}.

In Section \ref{S:Rec} we give exact recursions for the probabilities $\P(a\notto s),\P(s\notto b)$ and $\P(a\notto s, s\notto b)$. 
This is done by studying the size of the set of the vertices that can be reached along a directed path from a given vertex $s$ and 
the size of the set of vertices from which there is a directed path to $s$. 
We give recursions for the joint distribution of these sizes, which could be of independent interest.

The situation seems however to be much more complicated than what is suggested by Theorem \ref{T:korrp}. 
In Section \ref{S:Comp} we report our results from exact computations using the recursions. 
It turns out that for $n\ge 27$ the difference $\P(a\to s, s\to b)- \P(a\to s)\cdot \P(s\to b)$ seems to change sign three times. 
First twice for small $p$ (roughly $constant/n$) and then it goes from negative to positive again just before $p=1/2$. 
We find this behavior mysterious but conjecture it to be true in general and plan to study this further in a coming paper.

\medskip
For $0\le m\le \binom n2$, $\gnm$ is the random graph with $n$
vertices (which we take as $1,\dots,n$) and $m$ edges, uniformly
chosen among all such graphs. 
We let $\ggnm$ be the directed random graph obtained from $\gnm$ by
giving each edge  a random direction.
It is less straightforward to study the probabilities in $\ggnm$, but thanks
to Lemma \ref{L:Gnm}, we can prove the corresponding statement for $\ggnm$
in parallel with the proof for $\ggnp$. Here the critical probability is
$\approx 0.799288221$, see Theorem \ref{T:Gnm}. The exact computations and
the simulations for $\ggnm$ do not give reason to believe that there is
more than one critical probability for a fixed $n$. It might be
consternating that $\ggnm$ and $\ggnp$ behave so differently. In Section
\ref{S:Variance} we explain this and prove that the correlation in $\ggnp$
is always larger, in a certain sense, than in $\ggnm$. 
The difference is expressed as a variance for which we give an exact
asymptotic expression,  
see \refT {T:varcond}.

\medskip
The background of the questions studied in this paper is as follows. Let $G$ be any graph and $a,b,s,t$ distinct vertices of $G$. 
Further, assign a random orientation to the graph as described above, that is each edge is given its direction independent of the 
other edges. In \cite{SL2} it was proved that then the events $\{s\to a\}$ and $\{s\to b\}$ are positively correlated. 
This was shown to be true also if we first conditioned on $\{s\notto t\}$, i.e.
$\P(s\to a, s\to b\mid s\notto t )\ge \P(s\to a\mid s\notto t)\cdot \P(s\to
b\mid s\notto t)$. Note that it is not intuitively clear why this is so.  
It is for instance no longer true if we instead condition on $\{s\to t\}$. 
In another direction, it was proved that $\P(s\to b, a\to t\mid s\notto t )\le \P(s\to b\mid s\notto t)\cdot \P(a\to t\mid s\notto t)$. 
The proofs in \cite{SL2} relied heavily on the results in \cite{vdBK} and \cite{vdBHK} where similar statements were proved for 
edge percolation on a given graph. 
These questions have to a large extent been inspired by an interesting conjecture due to Kasteleyn named the Bunkbed conjecture 
by H\"aggstr\"om \cite{OH2}, see also \cite{SL1} and Remark 5 in \cite{vdBK}.

\medskip
\noindent
{\bf Acknowledgment:} We thank the anonymous referees for very helpful comments.

\section{Main Theorems} \label{S:Gnp}

Let $G(n,p)$ be the random graph where each edge has probability $p$ of
being present independent of other edges. We further orient each present
edge either way independently with probability $\frac{1}{2}$ and call the
resulting random directed graph $\ggnp$.  
It turns out to be convenient
to study the negated events $\{a \notto s\}$ and  $\{s\notto b\}$. We say that
two events $A$ and $B$ are negatively correlated if 
\[\Cov(A,B)\=\P(A\cap B)-\P(A)\P(B)<0,
\]
and positively correlated if the inequality goes the other way. If $A$ and $B$ are negatively correlated, then $A$ and $\neg B$ are positively correlated. Thus $a \notto s$ and $s \notto b$ are negatively correlated if and only if 
$\{a\to s\}$ and $\{s\to b\}$ are negatively correlated. Trivially, $\P(a \notto s)=\P(s \notto b)$.

All unspecified limits are as \ntoo.

The case $p=1$ corresponds to random orientations of the complete graph
$K_n$. In \cite{AL1} it was proved that $a \notto s$ and $s \notto b$ are
negatively correlated in $K_3$, independent in $K_4$ and positively
correlated in $K_n$ for $n\ge 5$. It was in fact proved that the relative
covariance $(\P(a \notto s, s \notto b)-\P(a \notto s)\P(s \notto b))/\P(a
\notto s, s \notto b)$ converged to $\frac 13$ as $n\to \infty$.

For $\ggnp$ our main theorem is as follows.

\begin{theorem}\label{T:korrp} For a fixed $p\in\hoi$ and fixed distinct
vertices $a,b,s\in\ggnp$ we have the following limit of the relative covariance
\[\lim_{n\to\infty} \frac{\P(a \notto s, s \notto b)-\P(a \notto s)\P(s \notto b)}{\P(a \notto s, s \notto b)}=\frac{2p-1}{3}.\]

In particular, for large enough $n$:
\begin{romenumerate}
  \item
If\/ $0<p<1/2$, then 
$\Cov(s\notto a ,\,b\notto s)<0$.
  \item
If\/ $1/2<p\le 1$, then 
$\Cov(s\notto a ,\,b\notto s)>0$.
\end{romenumerate}

Furthermore,
\begin{equation}
  \label{korrp}
\Cov(s\notto a ,\,b\notto s)
=(2 p-1+o(1))(1- p/2)^{2n-3}.
\end{equation}
\end{theorem}

One way to heuristically understand the critical probability $1/2$ is as
follows. For the event $\{s\notto a\}$ the two cases that dominate for large
$n$ are $i)$ no edge points out from $s$ and $ii)$ no edge points in to $a$,
each with probability $(1-p/2)^{n-1}$. Similarly for $\{b\notto s\}$. Thus
$\P(s\notto a)\P(b\notto s)$ is roughly $4\cdot (1-p/2)^{2n-2}$.
For the
joint probability $\P(s\notto a, b\notto s)$ on the other hand, we only need
to consider three of these cases since no edges at all around $s$ has a much
smaller probability. This gives a negative contribution to the covariance
and is perhaps the most striking contribution. But, the other three cases
have in the joint probability a factor $1-p/2$ less since one edge is
common, which gives a positive contribution. Thus the quotient
$\frac{\P(s\notto a, b\notto s)}{\P(s\notto a)\P(b\notto s)}\sim
\frac{3/4}{1-p/2}$, which is greater than 1 if and only if $p>1/2$. 

Theorem \ref{T:korrp} was proved in an earlier version of this manuscript,
see \cite{AL2}, with a rather lengthy calculation to estimate the involved
probabilities. In the present version we will present a shorter proof which
also has the advantage that it allows us to prove the covariance also in
$\ggnm$. However, we want to point out that the lengthy calculations in
\cite{AL2} could be shortened with a clever lemma by Backelin, see
\cite{B}.

For $\ggnm$, we obtain the corresponding result.

\begin{theorem}\label{T:Gnm}
  Suppose that $m,n\to\infty$, and
 that\/ $m/\nn\to p\in\hoi$. 
Let $\pc\approx 0.799$ be the unique root in $\oi$ of\/ 
$3 e^{-\frac{2p(1-p)}{(2-p)^2}}=4-2p$.
If $n$ is large enough, then for any
distinct vertices $s,a,b\in \ggnm$:
\begin{romenumerate}
  \item
If\/ $0<p<\pc$, then 
$\Cov(s\notto a ,\,b\notto s)<0$.
  \item
If\/ $\pc<p\le 1$, then 
$\Cov(s\notto a ,\,b\notto s)>0$.
\end{romenumerate}
\end{theorem}

In fact we have the following asymptotic formula for the covariance.

\begin{theorem}\label{T1}
  Suppose that $m,n\to\infty$, let
$p=p(n)\=m/\nn$, and assume $\liminf p(n)>0$. Then, for any
distinct vertices $s,a,b\in \ggnm$:
\begin{multline*}
\Cov(s\notto a ,\,b\notto s)
=\lrpar{\frac{3}{4-2p}\expQ{-\frac{2p(1-p)}{(2-p)^2}}-1+o(1)}
\P(s\notto a)^2
\\
=\left(3\cdot
e^{-\frac{2p(1-p)}{(2-p)^2}}-4+2p+o(1)\right)\cdot(1-p/2)^{2n-3}\cdot
e^{-\frac{2p(1-p)}{(2-p)^2}}. 
\end{multline*}
\end{theorem}

\section{Two lemmas}

In order to prove the theorems for $\ggnp$ and $\ggnm$ in parallel we need two lemmas.

\begin{lemma}
  \label{L0}
Fix $\ell\ge0$ edges in $K_n$. The probability that none of these
$\ell$ edges appears in \gnm{} is at most $\lrpar{1-m/\binom n2}^\ell$.
\end{lemma}
\begin{proof}
Let $N\=\binom n2$.
  We may assume that $m+\ell\le N$. Then the probability is
\begin{equation*}
  \frac{\binom {N-\ell}m} {\binom Nm}
=\frac{(N-\ell)!\,(N-m)!}{(N-\ell-m)!\,N!}
=\prod_{i=0}^{\ell-1} \frac{N-m-i}{N-i}
\le \prod_{i=0}^{\ell-1} \Bigpar{1-\frac{m}{N}}.
\qedhere
\end{equation*}
\end{proof}

Note that in $\gnp$ the corresponding probability is exactly $(1-p)^\ell$ by independence of the edges.
We next show a more precise version of \refL{L0} for the directed graph \ggnm.

Let $q(\ell)=q(\ell;n,m)$ be the probability that if we
fix $\ell$ edges in $K_n$ and give each an orientation, then none of
these directed edges appears in \ggnm. 
(Here $0\le \ell\le \binom n2$. 
By symmetry, the probability does not depend on the choice of the
$\ell$ edges and their orientations.) 
Let also $q'(\ell;n,p)$ be the corresponding probability in $\ggnp$. It is clear that 
 $q'(\ell;n,p)=(1-p/2)^\ell$. For $\ggnm$ we have the following result, which we believe is of
  independent interest. 
  
\begin{lemma}
  \label{L:Gnm}
Suppose that $\ell=\ell(n)=O(n)$ and that $0\le m=m(n)\le\binom
n2$. Then, with $p=p(n)=m(n)/\nn$, as \ntoo,
\begin{equation}\label{lgnm}
  q(\ell;n,m)\sim
(1-p/2)^\ell \exp\lrpar{-\Bigparfrac \ell n ^2\frac{p(1-p)}{(2-p)^2}}.
\end{equation}
Furthermore, for any $\ell,n,m$ we have $q(\ell;n,m)\le
q'(\ell;n,p)=(1-p/2)^\ell$.
\end{lemma}

\begin{proof}
Let $N\=\binom n2$. The cases $m=0$ and $m=N$ are trivial, so we may
assume that $0<m<N$ and thus $0<p<1$.

Let $X$ be the number of the $\ell$ chosen edges that appear in \gnm,
ignoring orientations. Then $X$ has the hypergeometric distribution
\begin{equation}\label{a2}
  \P(X=k)=\frac{\binom \ell k \binom{N-\ell}{m-k}}{ \binom{N}{m}}
=\binom\ell k \frac{(N-\ell)!\,(N-m)!\, m!}{N!\,(N-m-\ell+k)!\, (m-k)!}.
\end{equation}
Define, for $J,j\ge0$, $a(J,j)\=J^{-j}J!/(J-j)!$ (interpreted as 0
when $j>J$).
Then \eqref{a2} can be written as
\begin{equation*}
  \begin{split}
  \P(X=k)
&=\binom\ell k a(N,\ell)\qw N^{-\ell}
  a(N-m,\ell-k)(N-m)^{\ell-k}a(m,k)m^k	
\\
&=\binom\ell k p^k(1-p)^{\ell-k}a(N,\ell)\qw a(N-m,\ell-k)a(m,k).	
  \end{split}
\end{equation*}
Given $X=k$, the probability that none of the selected directed edges
appears in $\ggnm$ equals the probability that none of the $k$ selected
edges that appear in $\gnm$ gets the forbidden orientation, which
equals $(1/2)^k$.
Hence,
\begin{equation*}
  \begin{split}
q(\ell)
&=\sum_{k=0}^\ell 2^{-k}\P(X=k)
\\
&=\sum_{k=0}^\ell 
\binom\ell k \Bigparfrac p2^k(1-p)^{\ell-k}a(N,\ell)\qw a(N-m,\ell-k)a(m,k)
\\
&= a(N,\ell)\qw(1-p/2)^{\ell}\sum_{k=0}^\ell 
\binom\ell k \Bigparfrac{p/2}{1-p/2}^k\Bigpar{1-\frac{p/2}{1-p/2}}^{\ell-k} 
a(N-m,\ell-k)a(m,k)
\\
&= a(N,\ell)\qw(1-p/2)^{\ell}
\E\bigpar{a(m,Y)a(N-m,\ell-Y)},
  \end{split}
\end{equation*}
where $Y$ has the binomial distribution $\Bin(\ell,p/(2-p))$.
Consequently, \eqref{lgnm} is equivalent to 
\begin{equation}
  \label{a3}
a(N,\ell)\qw\E\bigpar{a(m,Y)a(N-m,\ell-Y)}
\sim
\exp\lrpar{-\Bigparfrac \ell n ^2\frac{p(1-p)}{(2-p)^2}}.
\end{equation}

In order to show \eqref{a3}, it suffices to show that every
subsequence has a subsequence where \eqref{a3} holds. (See \cite[p.\ 12]{JLR}.)
Consequently, by choosing suitable subsequences, we may assume that
\begin{equation}\label{a4}
  \ell/n\to\gl
\qquad\text{and}\qquad
p\to\pi,
\end{equation}
for some $\gl\in[0,\infty)$ and $\pi\in\oi$;
furthermore, we may assume that either
\begin{equation}\label{a4a}
  \ell\to\infty
\qquad\text{or}\qquad
l=O(1).
\end{equation}

We next observe that for all $J\ge1$ and $j\le J$,
\begin{equation*}
  a(J,j)=J^{-j}\prod_{i=0}^{j-1}(J-i)
=\prod_{i=0}^{j-1}\Bigpar{1-\frac iJ}
=\exp\lrpar{\sum_{i=1}^{j-1}\log\Bigpar{1-\frac iJ}}.
\end{equation*}
It follows that $0\le a(j,J)\le 1$ for all $j$ and $J$. Moreover, if
  $j=O(J\qq)$, then, if also $j\le J$,
\begin{equation*}
  \begin{split}
  {\sum_{i=1}^{j-1}\log\Bigpar{1-\frac iJ}}
&=
\sum_{i=1}^{j-1}\lrpar{-\frac iJ+O\Bigpar{\frac iJ}^2}
=-\frac{j^2-j}{2J}+O\Bigpar{\frac {j^3}{J^2}}
\\&
=-\frac{j^2}{2J}+O\bigpar{J\qqw}	
  \end{split}
\end{equation*}
and thus (also in the trivial case $j>J$)
\begin{equation}
  \label{a5}
a(J,j)=\expBig{-\frac{j^2}{2J}}+O\bigpar{J\qqw}.
\end{equation}

In particular, \eqref{a5} applies to $a(N,\ell)$, because
  $\ell=O(n)=O(N\qq)$, and thus
\begin{equation}
  \label{a5b}
a(N,\ell)
=\expbig{-\xfrac{\ell^2}{2N}}+o(1)
\to\expbig{-\gl^2},
\end{equation}
assuming \eqref{a4}.

Next, consider $a(m,Y)$, assuming \eqref{a4}--\eqref{a4a}.
Assume first $\pi>0$. If also $\ell\to\infty$, then by the law of
large numbers, $Y/\ell\pto \pi/(2-\pi)$, and thus
\begin{equation}\label{a6}
  \frac{Y^2}{2m}
=
\Bigparfrac Y\ell^2\Bigparfrac \ell n^2 \cdot\frac{n^2/2}{m}
\pto
\Bigparfrac \pi{2-\pi}^2 \gl^2 \frac1\pi
=\frac{\pi}{(2-\pi)^2}\gl^2.
\end{equation}
On the other hand,  if $\ell=O(1)$, then $\gl=0$; furthermore,
$Y\le\ell=O(1)$ and $m\to\infty$ (because $\pi>0$), so $Y^2/2m\pto0$; 
hence \eqref{a6} holds in this case too.

In the case $\pi=0$, so $p\to0$, we instead use
\begin{equation*}
  \E Y=\ell p/(2-p)\le p\ell =o\bigpar{p\qq\ell},
\end{equation*}
and thus $Y/(p\qq\ell)\pto0$. Hence,
\begin{equation*}
  \frac{Y^2}{2m}
=
\parfrac Y{p\qq\ell}^2\cdot\frac {p\ell^2}{2m}
=
\parfrac Y{p\qq\ell}^2\cdot\frac {\ell^2}{n(n-1)}
\pto
0
=\frac{\pi}{(2-\pi)^2}\gl^2,
\end{equation*}
so \eqref{a6} holds in this case too.

Consequently, \eqref{a6} holds in every case.
In particular, $Y=\Op(m\qq)$, and \eqref{a5} implies that,
provided $m\to\infty$,
\begin{equation}
  \label{a7a}
a(m,Y)=\explr{-\frac{Y^2}{2m}}+\op(1)
\pto\explr{-\frac{\pi}{(2-\pi)^2}\gl^2};
\end{equation}
on the other hand, if $m=O(1)$, then $\P(Y=0)\to1$, e.g.\ by \eqref{a6},
so \eqref{a7a} holds in this case too. Consequently, \eqref{a7a} holds in
every case.

Since $\ell-Y\sim\Bin\bigpar{\ell,1-p/(2-p)}=\Bin(\ell,2(1-p)/(2-p))$,
it follows by similar arguments, now considering the case $\pi=1$
separately, that
\begin{equation*}
  \frac{(\ell-Y)^2}{2(N-m)}
\pto
\frac{4(1-\pi)}{(2-\pi)^2}\gl^2
\end{equation*}
and
\begin{equation}
  \label{a7b}
a(N-m,\ell-Y)
\pto\explr{-\frac{4(1-\pi)}{(2-\pi)^2}\gl^2}.
\end{equation}

Combining \eqref{a7a} and \eqref{a7b} we find
\begin{equation*}
a(m,Y)a(N-m,\ell-Y)
\pto
\explr{-\frac{\gl^2}{(2-\pi)^2}\bigpar{\pi+4(1-\pi)}}.
\end{equation*}
By dominated convergence, using $0\le a(m,Y)a(N-m,\ell-Y)\le 1$, this, together with \eqref{a5b}, yields
\begin{equation*}
  \begin{split}
a(N,\ell)\qw\E\bigpar{a(m,Y)a(N-m,\ell-Y)}
&\to
\explr{\gl^2-\frac{\gl^2}{(2-\pi)^2}(4-3\pi)}
\\
=\explr{\frac{\gl^2}{(2-\pi)^2}(\pi^2-\pi)},
  \end{split}
\end{equation*}
which is equivalent to \eqref{a3} by \eqref{a4}.
Hence \eqref{a3} holds assuming \eqref{a4}--\eqref{a4a}, and thus in general,
which shows \eqref{lgnm}.

Finally, 
let $X'\sim\Bin(\ell,p)$ be the number of the $\ell$ chosen edges
that appear in $\gnp$, ignoring orientations. Then, 
$q(\ell)=\E{\xpar{\frac12}^X}$ as shown above, and similarly 
$q'(\ell)=\E{\xpar{\frac12}^{X'}}$. 
In fact, $\E a^X \le \E a^{X'}$ for any $a>0$, see
Hoeffding \cite[Theorem 4]{Hoeffding63}, 
which, taking $a=1/2$, yields $q(\ell)\le q'(\ell)$.
\end{proof}

\section{Proofs}

Say that a set $S\subseteq V(G)$ of vertices in a directed graph $G$
is an \emph{outset} 
if there is no directed edge $vw$ with $v\notin S$, $w\in S$
(\ie, all edges between $S$ and its complement are directed out from
$S$), and an \emph{inset} 
if there is no directed edge $vw$ with $v\in S$, $w\notin S$. Hence,
$S$ is an outset if and only if its complement is an inset.

\begin{lemma}
  \label{L:B1}
If $G$ is a directed graph and $a,b\in V(G)$, then 
the following are equivalent:
\begin{romenumerate}
\item $a\notto b$.
\item 
There exists an inset $S$ with $a\in S$, $b\notin S$.
\item 
There exists an outset $T$ with $a\notin T$, $b\in T$.
\end{romenumerate}
\end{lemma}

\begin{proof}
(i)$\iff$(ii):
  No directed path can leave an inset. Hence, if  an inset $S$ as in
  (ii) exists, then $a\notto b$. Conversely, if $a\notto b$, then
  $S\=\set{v\in V(G):a\to v}$ is an inset with $a\in S$, $b\notin S$.

(ii)$\iff$(iii):
Take $T\=V(G)\setminus S$, and conversely.
\end{proof}

For
$S\subseteq[n]$, let
$\cO_S$ and $\cI_S$ denote the events $\set{S \text{ is an outset}}$
and $\set{S \text{ is an inset}}$, respectively. The events will be in $\ggnm$ or $\ggnp$ depending on context.
We also write, for typographical convenience, $\cO_a\=\cO_{\set{a}}$, etc.

\begin{lemma}\label{L:saGnp}
Assume that $p\in (0,1]$ is fixed. Then, for any
distinct  $s,a\in \ggnp$,
  \begin{align*}
	\P(s\notto a)
&
=\P(\cI_s)+\P(\cO_a)+O\bigpar{n(1-p/2)^{2n}}
\\&
=2q'(n-1)+O\bigpar{n(1-p/2)^{2n}}
\\&
\sim2(1-p/2)^{n-1}.
  \end{align*}
\end{lemma}

 The proof of \refL{L:saGnp} is very similar to the proof of \refL{L:saGnm} below, but with $q(\ell;n,m)$ replaced by $q'(\ell;n,p)$.

\begin{lemma}\label{L:saGnm}
  Suppose that $m,n\to\infty$, let
$p=p(n)\=m/\nn$, and assume $\liminf p(n)>0$. Then, for any
distinct  $s,a\in \ggnm$,
  \begin{align*}
	\P(s\notto a)
&
=\P(\cI_s)+\P(\cO_a)+O\bigpar{n(1-p/2)^{2n}}
\\&
=2q(n-1)+O\bigpar{n(1-p/2)^{2n}}
\\&
\sim2(1-p/2)^{n-1}\expQ{-\frac{p(1-p)}{(2-p)^2}}.
  \end{align*}
\end{lemma}

\begin{proof}
Let $\cE_k$ be the event that there exists an inset $S$ with $s\in S$,
$a\notin S$ and $|S|=k$. Then, by \refL{L:B1},
\begin{equation}
  \label{b3}
\set{s\notto a} =\bigcup_{k=1}^{n-1}\cE_k
=\cE_1\cup\cE_{n-1}\cup\bigcup_{k=2}^{n-2}\cE_k.
\end{equation}
Here $\cE_1=\cI_s$ and $\cE_{n-1}=\cO_a$.

If $|S|=k$, then $\P(\cI_S)=\P(\cO_S)=q(k(n-k))$.
In particular, by \refL{L:Gnm}, 
\begin{equation}
  \label{b3c}
\P(\cI_s)=\P(\cO_{a})
=q(n-1)\sim(1-p/2)^{n-1}\expQ{-\frac{p(1-p)}{(2-p)^2}}.
\end{equation}
Further, for any $k$, there are $\binom{n-2}{k-1}$ sets $S$ with $s\in
S$, $a\notin S$ and $|S|=k$, and thus 
\begin{equation}
\P(\cE_k)\le \binom{n-2}{k-1}q(k(n-k))
\le {n}^{k-1}q(k(n-k)).  
\end{equation}

For fixed $k$ we may apply \refL{L:Gnm} and obtain
\begin{equation}\label{b3a}
\P(\cE_k)=\P(\cE_{n-k})
=O\bigpar{{n}^{k-1}(1-p/2)^{k(n-k)}}
=O\bigpar{{n}^{k-1}(1-p/2)^{kn}}.
\end{equation}
For a fixed $K$, we use this estimate for $k<K$. For $K\le k\le n-K$
we have $k(n-k)\ge K(n-K)$ and thus $q(k(n-k))\le q(K(n-K))$, while
the total number of subsets $S$ of $[n]$ is $2^n$, and thus \refL{L:Gnm}
again yields
\begin{equation}
  \label{b3b}
\P\Bigpar{\bigcup_{k=K}^{n-K}\cE_k}
\le 2^n q(K(n-K))
=O\bigpar{2^n(1-p/2)^{Kn}}.
\end{equation}
By the assumption $\liminf p(n)>0$, we may assume that $p=p(n)\ge p_0$
for some $p_0>0$.
Choosing $K$ such that $(1-p_0/2)^{K-2}\le1/2$, we have
$2(1-p/2)^K\le 2(1-p_0/2)^{K-2}(1-p/2)^2\le(1-p/2)^2$, and thus
\eqref{b3a} and \eqref{b3b} imply
\begin{equation*}
 \P\Bigpar{\bigcup_{k=2}^{n-2}\cE_k} 
=\sum_{k=2}^{K-1} O\bigpar{n^{k-1}(1-p/2)^{kn}} + O\bigpar{(1-p/2)^{2n}}
=O\bigpar{n(1-p/2)^{2n}}.
\end{equation*}
Hence, \eqref{b3} yields
\begin{equation}
  \label{b4}
\P\set{s\notto a} 
=\P(\cE_1\cup\cE_{n-1})+O\bigpar{n(1-p/2)^{2n}}
=\P(\cI_s\cup\cO_a)+O\bigpar{n(1-p/2)^{2n}}.
\end{equation}
Finally, $\P(\cI_s\cap\cO_a)=q(2n-3)=O\bigpar{(1-p/2)^{2n}}$ by
\refL{L:Gnm} again, and thus
$
\P(\cI_s\cup\cO_a)
=\P(\cI_s)+\P(\cO_a)+O\bigpar{(1-p/2)^{2n}}$.
The result now follows by \eqref{b4}
and \eqref{b3c}.
\end{proof}

\begin{lemma}\label{L:sabGnp}
  Suppose that $p\in (0,1]$ 
  is fixed, then for any
distinct  $s,a,b\in \ggnp$,
  \begin{align*}
\P(s\notto a &\text{ and } b\notto s)
\\&
=\P(\cI_{s,b}\cap \cI_b)
+\P(\cO_{s,a}\cap \cO_a)
+\P (\cO_a \cap \cI_b)
+o\bigpar{(1-p/2)^{2n}}
\\&
=3(1-p/2)^{2n-3}+o\bigpar{(1-p/2)^{2n}}
  \end{align*}
\end{lemma}

The proof of \refL{L:sabGnp} is almost
identical to that of \refL{L:sabGnm} below.

\begin{lemma}\label{L:sabGnm}
  Suppose that $m,n\to\infty$, let
$p=p(n)\=m/\nn$, and assume $\liminf p(n)>0$. Then, for any
distinct  $s,a,b\in [n]$,
  \begin{align*}
\P(s\notto a &\text{ and } b\notto s)
\\&
=\P(\cI_{s,b}\cap \cI_b)
+\P(\cO_{s,a}\cap \cO_a)
+\P (\cO_a \cap \cI_b)
+o\bigpar{(1-p/2)^{2n}}
\\&
=3q(2n-3)+o\bigpar{(1-p/2)^{2n}}
\\&
\sim3(1-p/2)^{2n-3}\expQ{-4\frac{p(1-p)}{(2-p)^2}}.
  \end{align*}
\end{lemma}

\begin{proof}
As in \eqref{b3}, we have 
\begin{equation*}
\set{s\notto a} =\bigcup_{k=1}^{n-1}\cE_k \quad\text{with}\quad
\cE_k\=\bigcup_{S:s\in S,\, a\notin S,\, |S|=k}\cI_S.  
\end{equation*}
and similarly
\begin{equation*}
\set{b\notto s} =\bigcup_{k=1}^{n-1}\cE'_k \quad\text{with}\quad
\cE'_k\=\bigcup_{S:s\in S,\, b\notin S,\, |S|=k}\cO_S.  
\end{equation*}

First, by the argument in the proof of \refL{L:saGnm} (choosing $K$ large
enough),
\begin{equation*}
 \P\Bigpar{\bigcup_{k=3}^{n-3}\cE_k} 
 =\P\Bigpar{\bigcup_{k=3}^{n-3}\cE'_k} 
=O\bigpar{n^2(1-p/2)^{3n}}
=o\bigpar{(1-p/2)^{2n}}.
\end{equation*}
We may thus approximate the event
\set{s\notto a \text{ and } b\notto s} by
\begin{multline*}
 ( \cE_1\cup\cE_2\cup\cE_{n-2}\cup\cE_{n-1})
\cap 
 ( \cE'_1\cup\cE'_2\cup\cE'_{n-2}\cup\cE'_{n-1})
\\
=\lrpar{\bigcup_{x\neq a} \cI_{s,x}\cup\bigcup_{y\neq s} \cO_{a,y}}
\cap
\lrpar{\bigcup_{z\neq b} \cO_{s,z}\cup\bigcup_{w\neq s} \cI_{b,w}}
\end{multline*}
(taking notational advantage of $I_s=I_{s,s}$ and so on).
This can be expanded as the union of $O(n^2)$ events of the type
$\cX_{S_1}\cap \cY_{\cS_2}$, 
where $\cX$ and $\cY$ are $\cI$ or $\cO$,
and $1\le|S_1|,|S_2|\le2$.

If $|S_1\cup S_2|\ge3$, 
then the event $\cX_{S_1}\cap \cY_{\cS_2}$
forbids at least 
$|S_1\cup S_2|(n-|S_1\cup S_2|)\ge 3(n-4)$ directed edges, and 
\refL{L:Gnm} shows that
\begin{equation}\label{bb}
\P(\cX_{S_1}\cap \cY_{\cS_2})=O\bigpar{(1-p/2)^{3n}};
\end{equation}
thus all such
combinations have together a probability 
$O\bigpar{n^2(1-p/2)^{3n}}=\opnn$ and
may be ignored.

Furthermore, any event $\cI_{s,x}\cap\cO_{s,z}$ forbids at least $n-3$
edges (regardless of orientation) in $\gnm$, and thus, by \refL{L0},
$$\P(\cI_{s,x}\cap\cO_{s,z})\le (1-p)^{n-3}.
$$
Since $p\ge p_0>0$, we have 
\begin{equation*}
  \frac{(1-p/2)^2}{1-p}
=  \frac{1-p+p^2/4}{1-p}
\ge 1+p_0^2/4,
\end{equation*}
and thus 
$n^2(1-p)^{n-3}
=o\bigpar{(1-p/2)^{2(n-3)}}
=\opnn$, so may also ignore all
$\cI_{s,x}\cap\cO_{s,z}$. We may similarly ignore
$\cO_a\cap \cI_{b,a}$,
$\cO_{a,b}\cap \cI_{b}$ and
$\cO_{a,b}\cap \cI_{b,a}$.

This leaves only the events 
$\cI_{s,b}\cap\cI_b$,
$\cI_{s}\cap\cI_b$,
$\cO_{a}\cap\cI_b$,
$\cO_{a}\cap\cO_s$,
$\cO_{a}\cap\cO_{s,a}$.
Of these, $\cI_{s}\cap\cI_b \subseteq \cI_{s,b}\cap\cI_b$
and $\cO_{a}\cap\cO_s \subseteq\cO_{a}\cap\cO_{s,a}$.
Summarizing, we have found
\begin{equation*}
  \P(s\notto a \text{ and } b\notto s)
=\P\bigpar{
(\cI_{s,b}\cap \cI_b)
\cup (\cO_{s,a}\cap \cO_a)
\cup (\cO_a \cap \cI_b)}
+o\bigpar{(1-p/2)^{2n}}.
\end{equation*}
The intersection of any two of the events
$\cI_{s,b}\cap \cI_b$, $\cO_{s,a}\cap \cO_a$ and $\cO_a \cap \cI_b$
is contained in an event
$\cX_{S_1}\cap \cY_{\cS_2}$
with $|S_1\cup S_2|=3$, so by \eqref{bb}, its probability is $\opnn$. 
Hence the first equality in the statement follows.

Moreover, each of these three events forbids exactly $2n-3$ directed edges,
and thus each has the probability $q(2n-3)$. The result follows by
\refL{L:Gnm}. 
\end{proof}

Note the obvious identity
$\P(s\notto a)=\P( b\notto s)$.

\begin{proof}[Proof of \refT{T:korrp}]
Follows easily from Lemmas \refand{L:saGnp}{L:sabGnp}.
\end{proof}
\begin{proof}[Proof of \refT{T1}]
An immediate consequence of Lemmas \refand{L:saGnm}{L:sabGnm}.
\end{proof}

\begin{proof}[Proof of \refT{T:Gnm}]
  Let $f(p)\=3 e^{-\frac{2p(1-p)}{(2-p)^2}}-4+2p$.
Note that $f(0)=-1$ and $f(1)=1$. We will show that $f'(p)>0$ for
$p\in\oi$; this implies the existence of a unique zero $\pc\in\oi$ of
$f$, with $f(p)<0$ for $0\le p<\pc$ and $f(p)>0$ for $\pc<p\le1$.
The result now follows from \refT{T1}. The numerical value
$\pc\approx0.799288221$ is found by \maple.

We use the partial fraction expansion
$-p(1-p)/(2-p)^2=1-3(2-p)\qw+2(2-p)\qww$ to find
\begin{equation*}
  f'(p)=6e^{-\frac{2p(1-p)}{(2-p)^2}}
\Bigpar{ -\frac3{(2-p)^2}+\frac4{(2-p)^3}}
+2.
\end{equation*}
The factor $-\xfrac3{(2-p)^2}+\xfrac4{(2-p)^3}$ is increasing on
$\oi$, as shown by another differentiation, and is thus at least
$-3/4+4/8=-1/4$. Since $0<e^{-\frac{2p(1-p)}{(2-p)^2}}\le1$, it
follows that $f'(p)\ge -6/4+2=1/2>0$ for $p\in\oi$.
\end{proof}

\begin{remark} \label{R:expfast} 
In the proofs of Lemmas \refand{L:saGnp}{L:sabGnp}, all error terms are
exponentially small compared to the main terms. Consequently, 
$\Cov(s\notto a ,\,b\notto s)/(1- p/2)^{2n-3}$
approaches its limit $2p-1$ exponentially fast for any fixed $p>0$
(this further holds uniformly for $p\ge p_0$ for any $p_0>0$),
and similarly,  
the relative covariance 
approaches its limit in \refT{T:korrp} exponentially fast for any fixed $p>0$.

Note, however, that \eqref{korrp} is false in the trivial case $p=0$, which
shows that the convergence cannot be uniform for all $p>0$.
\end{remark}
\begin{problem}
It would be interesting to know the rate of convergence in $G(n,m)$.  
\end{problem}

\section{Exact recursions in $\ggnp$} \label{S:Rec}
In this section we will give exact recursions to compute 
\[ \fpn:=\P_{\ggnp}(s\notto b) \text{ and }
\gpn:=\P_{\ggnp}(a\notto s, s\notto b).\] 

For a vertex $v\in V(G)$, let $\hpil{C}_v\subseteq V(G)$ be the (random) set
of all vertices $u$ for which there is a directed path from $v$ to $u$. We
will call this the {\bf out-cluster} from $v$. Let also the {\bf
  in-cluster}, $\vpil{C}_v\subseteq V(G)$ be the (random) set of all vertices
$u$ for which there is a directed path from $u$ to $v$. Note that we will
use the convention that $v\in \vpil{C}_v\cap\hpil{C}_v$. 
Let $y:=1-p/2$ 
be the probability that an edge does not exist with
a certain direction, and let $q:=1-p$ be
the probability that there is no edge
at all.

For $n\ge 1$, 
$s\in {\X}\subseteq [n]$ and $|{\X}|=k$ define:
\[ d_p(n,k):=\P_{\ggnp}(\hpil{C}_s={\X}),\]
where in particular $d_p(1,1)=1$.

\begin{lemma} \label{L:f}  We have the following recursions
\begin{romenumerate}
\item\label{lf1} 
$\displaystyle d_p(n,k)=d_p(k,k)y^{k(n-k)}, \quad \text{for $n>k\ge 1$}$,

\item\label{lf2}  
$\displaystyle 
d_p(k,k)=1-\sum_{i=1}^{k-1}\binom{k-1}{i-1}d_p(i,i)y^{i(k-i)},$ and

\item\label{lf3} 
$\displaystyle 
\fpn=\sum_{k=1}^{n-1}\binom{n-2}{k-1}d_p(k,k)y^{k(n-k)}$.
\end{romenumerate}
\end{lemma}
\begin{proof}
If $n>k$ there is a vertex $w\notin {\X}$. The only restriction on $w$ is
that it must not have any edge directed from ${\X}$ so
$\P_{\ggnp}(\hpil{C}_s={\X})=\P_{G(n-1,p)}(\hpil{C}_s={\X})\cdot y^{k}$ and
\ref{lf1} follows by induction. Clearly $\sum_{\X:s\in {\X}\subseteq [n]}
\P_{\ggnp}(\hpil{C}_s={\X})=1$, which gives formula \ref{lf2} after using
equation \ref{lf1}. To get the third equation we sum over all sets ${\X}$
that contain $s$ 
but not $b$, i.e. $\fpn=\sum_{{\X}}
\P_{\ggnp}(\hpil{C}_s={\X})=\sum_{k=1}^{n-1}\binom{n-2}{k-1}d_p(n,k)$, which
using \ref{lf1} gives \ref{lf3}. 
\end{proof}
Note that, by symmetry, also $\P_{\ggnp}(\vpil{C}_s={\X})=d_p(n,k)$.
\medskip

\begin{remark}
Note that by Lemma 2.1 in \cite{SL2}, which is a special case of a theorem by McDiarmid, \cite{CM}, 
our recursion for $d_p(k,n)$ also gives a formula for the probability that a
given set of vertices ${\X}$ 
with $|{\X}|=k$ is the connected component (or open cluster) containing $s$ in $G(n,p/2)$.
\end{remark}

We now want to do something similar for the more complicated case of $\gpn$.
For $n\ge 1$,
$ s\in {\X}\subseteq [n]$, $ s\in {\Y}\subseteq [n]$ with $|{\X}|=k,
|{\Y}|=m$ and  
$|[n]\setminus ({\Y}\cup {\X})|=r$ define:
\[ M_p(n,k,m,r):=\P_{\ggnp}(\hpil{C}_s={\X}, \vpil{C}_s={\Y}),\]
where in particular $M_p(1,1,1,0)=1$.

\begin{lemma} \label{L:g} 
We have the following recursions for $M_p$,
where $k+m>n-r\ge k,m$ and $k,m \ge 1$
\begin{romenumerate}
\item\label{lg1} 
$\displaystyle 
M_p(n,k,m,r)=M_p(n-r,k,m,0)q^{r(r+k+m-n)}y^{r(2n-2r-k-m)}, \quad \text{for $r>0$}$,
\item\label{lg2} 
$\displaystyle 
M_p(n,k,m,0)=\sum_{j=1}^{n-k}\binom{n-k-1}{j-1}M_p(n-j,k,m-j,0)d_p(j,j)y^{j(n-m)}\cdot$\\
$\displaystyle 
\big(y^{m+k-n}-y^{n-k-j}q^{m+k-n}\big) q^{(j-1)(m+k-n)}y^{(j-1)(n-k-j)},\quad \text{for $n>k, n\ge m$}$,
\item\label{lg3} 
 $\displaystyle M_p(n,k,m,r)=M_p(n,m,k,r)$,
\item\label{lg4} 
  $\displaystyle 
M_p(n,n,n,0)=$\\
 $\displaystyle {}\qquad
1-\sum_{j=1}^{n-1}\binom{n-1}{j-1}\sum_{k=j}^{n}\binom{n-j}{k-j}\sum_{m=j}^{n-k+j}\binom{n-k}{m-j}M_p(n,k,m,n-m-k+j),$ 
\item\label{lg5} 
 $\displaystyle
\gpn=\sum_{j=1}^{n-2}\sum_{k=j}^{n-1} \sum_{m=j}^{n-k+j} \binom{n-1}{j-1,\,m-j,\, k-j,\, n-k-m+j}\cdot$ \\
{\hbox to 5 em {}}  
$\displaystyle\frac{(n-k-1)(n-m-1)+n-j-1}{(n-1)(n-2)}M_p(n,k,m,n-m-k+j)$.
\end{romenumerate}
\end{lemma}

\begin{proof} Assume, as given for the first equation, that $r>0$. All
 vertices in $[n]\setminus ({\Y}\cup {\X})$ must not have any edge directed
 to ${\Y}$ or from ${\X}$. This means that there must be no edge at all to
 ${\Y}\cap {\X}$, which gives probability $q^{|[n]\setminus {\Y}\cup
 {\X}|\cdot |{\Y}\cap {\X}|}=q^{r(r+k+m-n)}$. There must not be any edge 
directed to $({\Y}\setminus {\X})$  and there must not be any edge directed from $({\X}\setminus {\Y})$. This gives a factor of $y^{|[n]\setminus {\Y}\cup {\X}|\cdot |({\Y}\setminus {\X})\cup ({\X}\setminus {\Y})|}=
 y^{r(2n-2r-k-m)}$.
 
 For equation \ref{lg2}, note that $n>k$ and $r=0$ imply
that there exist a vertex $w\in {\Y}\setminus {\X}$. Let $G$ be any directed
graph on $n$ vertices with $\hpil{C}_s={\X}$ and $\vpil{C}_s={\Y}$. If we
remove vertex $w$ and all its edges from $G$ the resulting graph will still
have $\hpil{C}_s={\X}$ since $w\notin {\X}$, whereas $\vpil{C}_s={\Y}'$, for
some $\Y '$ such that ${\Y}\cap {\X}\subseteq {\Y}'\subseteq {\Y}\setminus
\{w\}$. Let $j=|{\Y}\setminus {\Y}'|$ and sum over all possible ${\Y}'$.  
 The probability is $M_p(n-j,k,m-j,0)$ that the subgraph on $[n]\setminus
 ({\Y}\setminus {\Y}')$ is as needed. The subgraph on ${\Y}\setminus {\Y}'$
 must have $\vpil{C}_w={\Y}\setminus {\Y}'$ which gives probability
 $d_p(j,j)$.  
 There must not be any edge directed from ${\X}\setminus {\Y}$ to ${\Y}\setminus {\Y}'$, since the vertices of the latter do not belong to ${\X}$. The other direction is legal and this gives the factor $y^{(n-m)j}$. 
 There must by the definition of ${\Y}'$ not be any edge at all between ${\X}\cap {\Y}$ and ${\Y}\setminus ({\Y}'\cup \{w\})$, which gives the factor $q^{(k+m-n)(j-1)}$. 
 There must also not be any edge directed from  ${\Y}\setminus ({\Y}'\cup
 \{w\})$ to ${\Y}'\setminus ({\Y}\cap {\X})$, which gives the factor
 $y^{(n-k-j)(j-1)}$. The only possible edges left to consider have one
 endpoint in $w$ and the other in $\Y'$.
 The edges between ${\Y}'\setminus ({\Y}\cap {\X})$ and $w$ could have any direction and there must not be any edge directed from ${\Y}\cap {\X}$ to $w$ (since $w\notin {\X}$), but there must be at least one edge directed from $w$ to ${\Y}'$ (since $w\in {\Y}$). This gives probability $(y^{m+k-n}\cdot 1^{n-k-j}-q^{m+k-n}\cdot y^{n-k-j})$. Putting all this together gives formula \ref{lg2}.

The third equation is obtained from the symmetry of reversing all directions.

The fourth equation follows from the fact that 
\[\sum_{\X,\Y :s\in {\Y},{\X}\subseteq [n]} \P_{\ggnp}(\hpil{C}_s={\X},\vpil{C}_s={\Y})=1.\] 
Here $j=|{\Y}\cap {\X}|$ and recall that $s\in {\Y}\cap {\X}$ is a necessary condition.

The last equation is obtained by summing over all possible pairs ${\Y},{\X}$
such that $a\notin {\Y}, b\notin {\X}$. Again $j=|{\Y}\cap {\X}|$ and the
formula is split into the cases when $a\notin {\X}$ and $a\in {\X}$,
respectively. We get 
\begin{multline*}
  \gpn=\sum_{j=1}^{n-2}\binom{n-3}{j-1}\cdot \\
\qquad  \Biggl( \;
\sum_{k=j}^{n-2}\binom{n-2-j}{k-j}\sum_{m=j}^{n-k+j-1}\binom{n-k-1}{m-j}M_p(n,k,m,n-m-k+j)\\
+\sum_{k=j+1}^{n-1}\binom{n-2-j}{k-j-1}\sum_{m=j}^{n-k+j}\binom{n-k}{m-j}M_p(n,k,m,n-m-k+j)
\Biggr),
\end{multline*}

which after simplification gives the claimed formula.
\end{proof}

Note that in $\ggnp$ the functions $\P(a \notto s)$ and $\P(a \notto s, s
\notto b)$  
are polynomials in $p$ and hence continuous.

\section{Computations and Conjectures for $\ggnp$} \label{S:Comp}
We have used \maple\ to compute the functions $\fpn$ and $\gpn$ for
$n\le 30$. Figure \ref{F:Plot} displays the relative covariance
$(\gpn-\fpn^2)/\gpn$. All curves start with being mildly
negative. They then turn positive and for $n<27$ they stay positive. For
$n\ge 27$ however, they go below the $p$-axis again for some time.

\begin{figure}[htbp]
\begin{center}
\epsfig{file=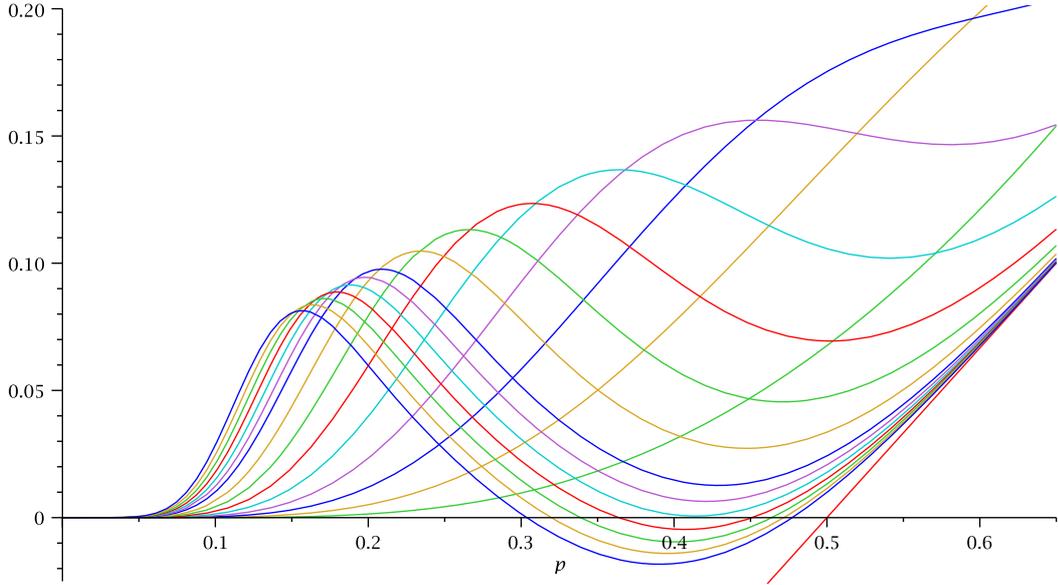, height=8cm}
\end{center}
\caption{The relative correlation $\frac{\P(a \notto s, s \notto b)-\P(a \notto s)\P(s \notto b)}{\P(a \notto s, s \notto b)}$ in $\ggnp$
for $n=8,10,12,14,16,18,20,22,24,25,26,27,28,29,30$, 
and the asymptote $(2p-1)/3$. 
All curves are negative for very small $p$. For $n\ge 27$ we get three zeros.} 
\label{F:Plot}
\end{figure}

For larger values of $n$ it becomes infeasible for our computers to obtain the exact functions. 
We have instead for various fixed values of $p$ used the recursions to obtain the value of the quotient $(\gpn-\fpn^2)/\gpn$ 
for $n\le 300$. Based on these calculations we conjecture the following.

\begin{conjecture}\label{C:3z}
For $n\ge 27$, the relative covariance changes sign at three critical probabilities $p_1(n)<p_2(n)<p_3(n)$.
\end{conjecture}

\begin{conjecture} Asymptotically, $p_1(n)\sim c_1/n$ and $p_2(n)\sim c_2/n$
  for some constants $c_1$ and $c_2$.
\end{conjecture}

The computations indicate that very rough estimates of $c_1$ and $c_2$ are
$0.36$ and $7.5$, respectively. 

It follows from \refR{R:expfast} that there is a critical probability $p(n)$ 
such that $p(n)\to1/2$ exponentially fast; if our conjectures hold, this
critical probability is thus $p_3(n)$.

\begin{conjecture} \label{C:p3}For all $n\ge 27$, $p_3(n)<1/2$. 
\end{conjecture}

\begin{conjecture} \label{C:above} For $n\ge 8$, the relative covariance $\frac{\P(a\notto s, s\notto b)-\P(a\notto s)\P(s\notto b)}{\P(a\notto s, s\notto b)}> \frac{2p-1}{3}$.
\end{conjecture}

Note that Conjecture \ref{C:above} implies 
Conjecture \ref{C:p3} and the following related conjecture. 

\begin{conjecture} \label{C:enhalv} For $p=\frac 12$, and $n\ge 6$ the relative covariance $\frac{\P(a\notto s, s\notto b)-\P(a\notto s)\P(s\notto b)}{\P(a\notto s, s\notto b)}$ is positive.
\end{conjecture}

\section{Exact recursions in $\ggnm$} \label{S:Recm}

For convenience, let $A:=\{a\notto s\}$ and $B:=\{s\notto b\}$.
In this section we will derive recursions for $\P(A)$ and $\P(A\cap B)$ in $\ggnm$ using the corresponding exact recursions in $\ggnp$ from Section \ref{S:Rec}.

Let as above
$f_n(p):=\P_{\ggnp}(A)=\P_{\ggnp}(B)$ and $g_n(p):=\P_{\ggnp}(A\cap B)$, 
and let 
$h_n(m):=\P_{\ggnm}(A)=\P_{\ggnm}(B)$ and $k_n(m):=\P_{\ggnm}(A\cap B)$.
Further, let $N:=\binom n2$ be the number of edges in the complete graph $K_n$ and $M\sim\Bin(N,p)$ be the actual number of edges in $G(n,p)$.
Using that $G(n,m)$ can be seen as $\bigpar{G(n,p)\mid M=m}$, we can express the functions $f_n(p)$ and $g_n(p)$ for $\ggnp$ using the corresponding functions, 
$h_n(m)$ and $k_n(m)$, for $\ggnm$.

\begin{align*}
f_n(p)&=\sum_{m=0}^{N}\P(M=m)\cdot h_n(m)=\sum_{m=0}^{N}\binom{N}m p^m (1-p)^{N-m}\cdot h_n(m),\\
g_n(p)&=\sum_{m=0}^{N}\P(M=m)\cdot k_n(m)=\sum_{m=0}^{N}\binom{N}m p^m (1-p)^{N-m}\cdot k_n(m).
\end{align*}
These relations can be inverted by repeated differentiating.
\begin{theorem}\label{T:rec}
\begin{align*}
h_n(m)&=\frac{(N-m)!}{N!}f_n^{(m)}(0)-\sum_{i=0}^{m-1}\binom mi (-1)^{m-i}\cdot h_n(i),\\
k_n(m)&=\frac{(N-m)!}{N!}g_n^{(m)}(0)-\sum_{i=0}^{m-1}\binom mi (-1)^{m-i}\cdot k_n(i).
\end{align*}
\end{theorem}

\begin{proof}
It is sufficient to show one of the recursions.
\begin{align*}
f_n(p)&=\sum_{i=0}^{N}\binom{N}i p^i (1-p)^{N-i}\cdot h_n(i)\\
&=\sum_{i=0}^{N}\binom{N}i \sum_{j=0}^{N-i}\binom{N-i}{j}p^{i+j}(-1)^j\cdot h_n(i)\\
&=\sum_{k=0}^{N}\frac{N!}{(N-k)!}p^k\sum_{i=0}^k\frac1{i!(k-i)!}(-1)^{k-i}\cdot h_n(i).
\end{align*}
Differentiating $m$ times and inserting $p=0$ gives
\begin{align*}
f_n^{(m)}(0)&=\frac{N!}{(N-m)!}m!\sum_{i=0}^m\frac1{i!(m-i)!}(-1)^{m-i}\cdot h_n(i)\\
\end{align*}
from which we get
\[\sum_{i=0}^m\binom mi(-1)^{m-i}\cdot h_n(i)=\frac{(N-m)!}{N!}f_n^{(m)}(0),\]
which, after rearranging, gives the desired recursion.
\end{proof}

Computer calculations based on these recursions lead us to the following conjecture.
\begin{conjecture}\label{C:1z}
For any fixed $n\ge 5$, the covariance in $\ggnm$ changes sign only once between two values of $m$.
\end{conjecture}

 \section{Comparison between $\ggnp$ and $\ggnm$} \label{S:Variance}
As before, let $A:=\{a\notto s\}$ and $B:=\{s\notto b\}$.
For moderate $n$, the correlation between $A$ and $B$ is positive in $\ggnp$  
for quite small $p$, while, for $\ggnm$ the proportion of links, $m/N$, where $N=\binom n2$,
needs to be closer to 1 to get a positive correlation.
In fact, we may study the conditional covariance given $M\sim\Bin(N,p)$ to show that the covariance 
in $\ggnp$ exceeds the average covariance in 
$\ggnm$ for fixed $n$ and $p$.

\begin{fact} \label{F:Compare}
  \begin{equation}
	\label{f81}
\Cov_{\ggnp}(A,B)=\E(\Cov(A,B\mid M))+\Var(\P(A\mid M))
  \end{equation}
where
\[
\E(\Cov (A,B\mid M))=\sum_{m=0}^{N}\P(M=m)\cdot \Cov_{\ggnm}(A,B).\]
\end{fact}

To understand this statement (a standard type of variance analysis)
we will as in Section \ref{S:Recm} 
view $\gnm$ as $(\gnp\mid M=m)$.  Note that 
\begin{align*}
\Cov_{\ggnp}(A,B)&=\P_{\ggnp} (A\cap B)-\P_\ggnp (A)\P_\ggnp (B)\\
&=\E (\P(A\cap B\mid M))-\E (\P(A\mid M)\P(B\mid M)) \\
& \hskip 4em 
+ \E (\P(A\mid M)\P(B\mid M)) -\E (\P(A\mid M))\E (\P(B\mid M))\\ 
&=\E(\Cov (A,B\mid M))+\Cov(\P(A\mid M),\P(B\mid M))
\end{align*}
and, as $\P(A\mid M)=\P(B\mid M)$, the formula follows.

The left-hand side of \eqref{f81} is 
$\sim(2 p-1)(1- p/2)^{2n-3}$
by \refT{T:korrp}.
We can obtain the asymptotics of the two terms on the
right-hand side too.

\begin{theorem}  \label{T:varcond}
For every fixed $p\in(0,1]$,
  \begin{align*}
\E(\Cov(A,B\mid M))&=
\Bigpar{3-(4-2 p)e^{2\frac{p(1-p)}{(2-p)^2}}+o(1)}
\bigpar{1-\xfrac p2}^{2n-3} ,
\\
\Var(\P(A\mid M))
&= 4
\Bigpar{e^{2\frac{p(1-p)}{(2-p)^2}}-1+o(1)}
\bigpar{1-\xfrac p2}^{2n-2} 
.
  \end{align*}
\end{theorem}
Note that all three terms in \eqref{f81}, as well as $\P(A)\P(B)=\P(A)^2$, 
are of the same order (except when a term vanishes), 
 viz.\  $(1-p/2)^{2n}$,
see \refT{T:korrp} and \refL{L:saGnp}.

\begin{proof}
The sum of the right-hand sides equals
$(2 p-1+o(1))(1- p/2)^{2n-3}$,
so by \eqref{korrp} and \eqref{f81}, it suffices to prove the second formula.

Let $\hp=M/N=M/\binom n2$, and note that by the Law of Large Numbers,
$\hp\pto p$.
We begin by observing that by the Central Limit Theorem, 
$$
n(\hp-p)
=\frac{n}{N\qq}\cdot\frac{M-Np}{N\qq}
\dto \sqrt{2p(1-p)}Z,
$$
where $\dto$ denotes convergence in distribution and $Z\sim N(0,1)$ is a
standard normal variable.
It follows that for any real constants $a$ and $b$,
\begin{equation}
  \label{jp}
\frac{(1-\hp/2)^{an+b}}{(1-p/2)^{an+b}}
=\lrpar{1-\frac{\hp-p}{2-p}}^{an+b}
\dto \exp\lrpar{a\frac{\sqrt{2p(1-p)}}{2-p}Z}.
\end{equation}

By \refL{L:saGnm},
  \begin{equation*}
	\begin{split}
\frac{\P(A\mid M)}{(1-\hp/2)^{n-1}}
&= 2 e^{-\hpp}  +\op(1)
= 2  e^{-\pp} +\op(1),
	\end{split}
  \end{equation*}
and thus by \eqref{jp}
  \begin{equation}\label{ma}
\frac{\P(A\mid M)}{(1-p/2)^{n-1}}
\dto {2  e^{-\pp}}
\exp\lrpar{\frac{\sqrt{2p(1-p)}}{2-p}Z}.
  \end{equation}
Denote the right-hand side of \eqref{ma} by $R$. 
Since $\E e^{cZ}=e^{c^2/2}$ for any real $c$, we have
$$
\E R^r=2^r e^{(r^2-r)\pp}, 
$$
in particular, $\E R^2=4 e^{2\pp}$ and $\E R=2$,
so $\Var R=4 e^{2\pp}-4$. Hence the result follows if the variance converges
in \eqref{ma}. For this, it suffices to show that
\begin{equation}\label{er3}
\E\lrpar{\frac{\P(A\mid M)}{(1-p/2)^{n-1}}}^4 =O(1).  
\end{equation}
(See e.g.{}
\cite[Theorems 5.4.2 and 5.4.9]{Gut} for this standard argument, and note that
the same argument shows that all moments converge in \eqref{ma}.)

To verify \eqref{er3}, let $p_0=p/2$. 
The proof of \refL{L:saGnm} shows that the estimates in \refL{L:saGnm} 
(with $p$ replaced by $\hp$) hold uniformly for $\hp\ge p_0$.
Since $q(n-1;n,m)\le q'(n-1;n,\hp)=(1-\hp/2)^{n-1}$ 
by \refL{L:Gnm}, this yields 
$\P(A\mid M)=O\bigpar{(1-\hp/2)^{n-1}}$, provided $\hp\ge p_0$.
For $\hp<p_0$, we simply use $\P(A\mid M)\le1$. Hence, for some constant
$C$,
\begin{equation}
\E\lrpar{\frac{\P(A\mid M)}{(1-p/2)^{n-1}}}^4 
\le (1-p/2)^{-4(n-1)}\P(\hp<p_0)
+C\E\lrpar{\frac{1-\hp/2}{1-p/2}}^{4(n-1)}. 
\end{equation}
The first term on the right-hand side is bounded since
$\P(\hp< p_0)=\P(M<Np_0)\le e^{-Np/8}$ by the Chernoff bound 
\cite[Theorem  2.1]{JLR}. For the second term we have,
letting $h=8/(n(2-p))$ and recalling that $M$ is a sum of
$N$ independent copies of $I\sim\Bin(1,p)$,
\begin{equation*}
  \begin{split}
\E\lrpar{\frac{1-\hp/2}{1-p/2}}^{4(n-1)}
&=\E\lrpar{1-\frac{\hp-p}{2-p}}^{4(n-1)}
\le\E\exp\lrpar{-4(n-1)\frac{\hp-p}{2-p}}
\\&
=\E\exp\bigpar{-h(M-Np)}	
=\lrpar{\E\exp\bigpar{-h(I-p)}}^N.	
  \end{split}
\end{equation*}
Now $\E e^{-h(I-p)}\le e^{h^2/8}$, see \cite[(4.16)]{Hoeffding63}
(a weaker estimate suffices), and thus the last expression is bounded by
$e^{Nh^2/8}=O(1)$. This shows \eqref{er3} and completes the proof.
\end{proof}

Numerical computations 
using the recursions of Sections \ref{S:Rec} and \ref{S:Recm} 
also suggest that 
all three quantities in  \eqref{f81}, normalized by e.g.\
$(1-p/2)^{2n-3}$, converge quickly unless $p$ is very small, 
cf.\ \refR{R:expfast}; 
see Figure \ref{F:Var} for the case $n=30$.

\begin{figure}[htbp]
\begin{center}
\epsfig{file=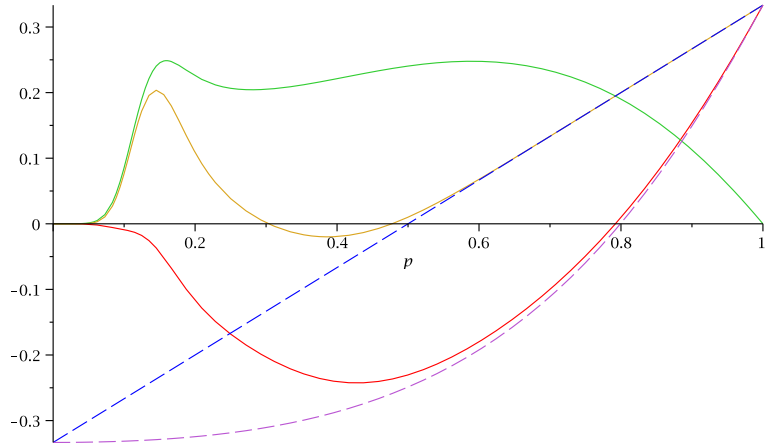, height=6cm}
\end{center}
\caption{The solid curves are the covariance for $\vec G(30,p)$, middle
curve, which by Fact \ref{F:Compare} is the sum of the expected covariance
for $\vec G(30,M)$, bottom, and $\Var(\P(A\mid M))$, top curve.
The dashed curves are the asymptotic curves for $\vec G(n,p)$ (straight
line) and expected covariance for 
$\vec G(n,M)$. All curves normalized with $3(1-p/2)^{2n-3}$.}
\label{F:Var}
\end{figure} 

 \section{Quenched version} \label{S:Quenched}
As mentioned in the introduction, we have so far studied the annealed model,
i.e.\ the joint probability space of $\ggnp$ (or $\ggnm$) and that of the
orientations. In the quenched model, the covariance is computed for each
graph of $\ggnp$ (or $\ggnm$) and then averaged over all graphs. 

It is quite common that the results differ between the two models, and this
seems to be the case here also. We have computed the quenched expectations
for $\ggnp$ (and $\ggnm$) for small $n$ ($n\le8$) and the covariances, as
functions of $p$, look quite different, as can be seen from Figure
\ref{F:Plot-qa} for $n=8$.

\begin{figure}[htbp]
\begin{center}
\epsfig{file=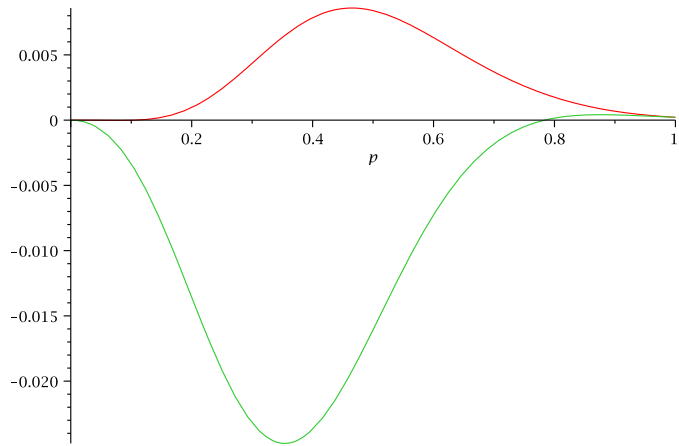, height=6cm}
\end{center}
\caption{Annealed (top) and quenched (bottom) covariances for $\vec G(8,p)$.}
\label{F:Plot-qa}
\end{figure}

For $n=4,\dots,8$ there is only one zero for the covariance, but this differs dramatically from the zero of the annealed model, as can be seen from Table \ref{T:zeroes}.

\begin{table}[h]
\begin{center}
\begin{tabular}{ccc}
$n$&Annealed&Quenched\\
\hline
4&1.000&1.000\\
5&0.729&0.927\\
6&0.276&0.857\\
7&0.152&0.809\\
8&0.107&0.783
\end{tabular}
\end{center}
\caption{Zeroes of the covariances for $\ggnp$ in the annealed and quenched models.}
\label{T:zeroes}
\end{table}

Conditioning on the graph and taking expectations, we get that, similarly to
Fact \ref{F:Compare},
\begin{equation}  \label{que}
\Cov_{\vec G_a(n,p)}(A,B)
 =\Cov_{\vec G_q(n,p)}(A,B)+\Cov(\P(A\mid G),\P(B\mid G)),
\end{equation}
where $\vec G_a(n,p)$ denotes the annealed model and $\vec G_q(n,p)$ denotes
the quenched model; recall that $\Cov_{\vec G_q(n,p)}(A,B)$ is defined as
$\E_{\vec G_q(n,p)}(\Cov(A,B\mid G))$.
(A similar formula holds for $\ggnm$ and the corresponding quenched model.)

Here the conditional probabilities for $A$ and $B$ given the graph need not
be equal, so that the last covariance in \eqref{que}
could possibly be negative for some values of $p$ and $n$.
Even though our computations show that this is not the case when $n\le 8$.

\begin{figure}[hbpt]
\begin{center}
\epsfig{file=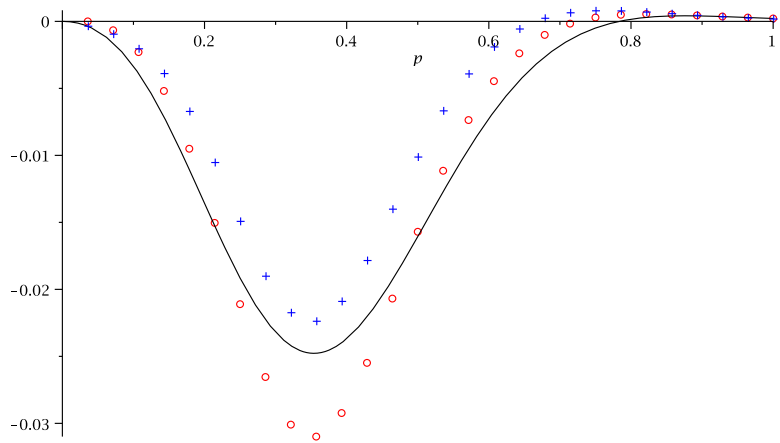, height=6cm}
\end{center}
\caption{Covariances of $A$ and $B$ for $\vec G(8,p)$ (solid) and $\vec G(8,m)$
  ($\circ$) in the quenched model and for $\vec G(8,m)$ ($+$) in the
  annealed model.} 
\label{F:Plot-q8}
\end{figure}

It is worth noting that, in contrast with the annealed model, in the
quenched model $\ggnp$ and $\ggnm$ behave similarly, see Figure
\ref{F:Plot-q8}; 
this is not surprising, since the average taken in the
quenched $\ggnp$ can be obtained by averaging over quenched $\ggnm$ with
suitable weights.
Also the quenched models appear, at least for small values of $n$, to be
much closer to annealed $\ggnm$ than annealed $\ggnp$.
In other words, the variation between different graphs with the same number
of edges is of less importance than the variation caused by different number
of edges. (The latter variation is quantified by   \refT{T:varcond}.)
This seems 
intuitively reasonable, since $\ggnm$ can be regarded as $\ggnp$ conditioned
on the number of edges, which thus can be seen as a ``semi-quenched'' version.
\begin{problem}
It would be interesting to find asymptotics for the quenched versions.
\end{problem}

\newcommand\webcite[1]{\hfil  
   \penalty0 
\texttt{\def~{{\tiny$\sim$}}#1}\hfill\hfill}
\newcommand\arXiv[1]{\webcite{arXiv:#1.}}

\end{document}